\documentclass[leqno]{amsart}
\usepackage{amsmath,amsfonts,amssymb,amsthm,amscd,latexsym}
\usepackage{cite}
\usepackage{color}
\usepackage{enumerate}
\numberwithin{equation}{section}
\usepackage{mathrsfs}
\usepackage{mathtools}
\usepackage{tikz}

\newtheorem{thm}{\bf Theorem}[section]
\newtheorem{lem}[thm]{\bf Lemma}

\newtheorem{rmk}[thm]{\bf Remark}

\newtheorem{cor}[thm]{\bf Corollary}

\newtheorem{defi}[thm]{\bf Definition}
\newtheorem{definition}[thm]{\bf Definition}
\newtheorem{prop}[thm]{\bf Proposition}

\newtheorem{quest}[thm]{\bf Question}
 
\usepackage{enumerate}
  \usepackage{comment}
\usepackage{hyperref}					
\hypersetup{colorlinks,
	linkcolor=blue,%
	citecolor=blue}
\newcommand{\norm}[1]{\left\lVert#1\right\rVert}
\newcommand{\cM}{\mathcal{M}}
\newcommand{\cE}{\mathcal{E}}
\newcommand{\cA}{\mathcal{A}}
\newcommand{\cN}{\mathcal{N}}

\newcommand{\cP}{\mathcal{P}}
\newcommand{\cR}{\mathcal{R}}
\newcommand{\cZ}{\mathcal{Z}}
\begin{document}

\title[Isomorphisms between symmetric spaces 
]{Isomorphisms between symmetric spaces 
over infinite and finite von Neumann algebras
}

\author[J. Huang]{Jinghao Huang}
\address{Institute for  Advanced Study in  Mathematics of HIT, Harbin Institute of Technology, Harbin 150001, China}
\email{{\color{blue}jinghao.huang@hit.edu.cn}}

\author[F. Sukochev]{Fedor Sukochev}

\author[D. Zanin]{Dmitriy  Zanin}
\address{School of Mathematics and Statistics, UNSW, Kensington, NSW 2052, Australia}
\email{\color{blue}f.sukochev@unsw.edu.au}
\email{\color{blue}d.zanin@unsw.edu.au}

\thanks{The first author  was supported by the NNSF of China  (No. 12031004, 12301160 and 12471134). The second and  the  third   authors were supported by ARC (DP230100434)}

\begin{abstract}
The primary aim of this paper is to show
a linear topological isomorphism between (elements of a wide 
class {of})  symmetric spaces over the hyperfinite $II_1$ factor 
$\mathcal{R}$ and certain symmetric operator space over the hyperfinite 
$II_\infty $ factor 
$\mathcal{R}\bar{\otimes}\mathcal{L}(H))$.
 Precisely, we show that 
 for any   symmetric function space $E(0,1)$ (in the sense of 
 Lindenstrauss and Tzafriri) such that both $E(0,1)$ and its K\"othe dual have the Kruglov property,  the
 symmetric operator space $E(\mathcal{R})$ is isomorphic to some symmetric space $Z_E^2(\mathcal{R}\bar{\otimes}\mathcal{L}(H))$. This result  establishes a noncommutative version of a well-known result due to Johnson, Maurey, Schechtman and Tzafriri,  and
 answers the noncommutative version of a question due to Mityagin.  
\end{abstract}



\maketitle

\section{Introduction}

In this paper, we are concerned with the following question:

\begin{quest}\label{question2}
Can a noncommutative  symmetric space $F(\mathcal{R}\bar{\otimes}\mathcal{L}(H))$ be isomorphic to some  noncommutative symmetric space $E(\mathcal{R})?$
Here, $\cR$ stands for the hyperfinite $II_1$ factor. 
\end{quest}

For factors of type I$_n$ and I$_{\infty}$ non-commutative symmetric spaces were introduced by von Neumann \cite{v43} and Schatten \cite{Schatten}, respectively. For a general semifinite von Neumann algebra $\mathcal{M},$ non-commutative symmetric spaces $E(\mathcal{M})$ were formally introduced by Ovchinnikov \cite{Ov} but they were lurking in the background yet in the pioneering work of Segal\cite{Se}. 

The unique role of the hyperfinite factors $\mathcal{R}$ and $\mathcal{R}\bar{\otimes}\mathcal{L}(H)$ in the theory of injective von Neumann algebras was demonstrated in the fundamental paper by Connes \cite{C76}.

$L_p(\mathcal{M})$ is the best known example of  a  non-commutative symmetric space.
Note that although $L_p(0,1)$ and $L_p(0,\infty )$ are isometric for any $1\leq p\leq \infty$, 
a similar result does not hold in the noncommutative realm.
In fact,  it was shown that any noncommutative $L_p$-space  $L_p(\cM)$, $1\le p\ne 2 <\infty $,  affiliated with a  semifinite infinite von Neumann algebra $\cM$ is not isomorphic to a subspace of  $L_p(\cN) $ affiliated with a finite von Neumann algebra $\cN$\cite{HRS}, see also \cite{Mc,S00,S96,S01,AL}.
On the other hand, 
if $1\le p<2$, then a noncommutative $L_p$-space $L_p(\cM)$ affiliated with a semifinite infinite von Neumann algebra $\cM$ can be isomorphically embedded into the $L_1(\cN)$ for some finite von Neumann algebra $\cN$, see e.g. \cite{Junge,Rand,JSZ}. 
Isomorphic embeddings between noncommutative  $L_p$-spaces and noncommutative symmetric spaces were studied in \cite{JSZ, HJSZ,HJSZ24}. 
However,   Question \ref{question2} has not been studied.

Prior to discussing our main results and underlying techniques, we recall the commutative antecedents of these results.

From its very inception \cite{B}, Banach space theory was concerned with the problem of classifying various classes of Banach spaces, in particular symmetric function spaces.
A very important role in the development of this theory was played by a question due to  
Mityagin \cite{Mityagin}, who asked whether a symmetric function space distinct from $L_p$-spaces, $1\leq p\leq \infty$,   may have representations both on finite and infinite intervals.
 \begin{quest}\label{question1}\cite[p.99]{Mityagin}
Can a  symmetric space $E$ of measurable functions on $(0, 1)$ other
than $L_p $, $ 1 \leq p < \infty $, be isomorphic to any symmetric
space $ F $ on the semi-axis $(0, \infty)$ or the axis $(-\infty, +\infty)$?
\end{quest}
This question is highly non-trivial. For example,  $L_{p,q}(0,\infty)$ is not isomorphic to $L_{p,q}(0,1)$ whenever $1<p <\infty$, $1\le q<\infty$ and $p\ne q $, see e.g. \cite{CD88,HS,SS18,KS15}.
Question \ref{question1} partly motivated serious developments of Banach space geometry,  recorded in two influential monographs
\cite{LT2} and \cite{JMST}. 
We explain now the main achievements presented there.
For a symmetric functon space $E(0,1)$, we denote by 
$Z_E^2(0,\infty )$  the space  of all measurable functions $f$ on 
$(0,\infty )$
 such that 
\begin{align}\label{defZE2}
\norm{f}_{Z_E^2} := \norm{\mu(f)\chi_{(0,1)}}_E+\norm{\mu(f) \chi_{(1,\infty)}}_2<\infty ,
\end{align}
where $\mu(f)$ stands for the decreasing rearrangement of $|f|$. 
The following result  is contained in \cite[Section 8]{JMST}, see also \cite[Theorem 2.f.1]{LT2} and \cite{AM}.

\begin{thm}\label{jmst}
Let $E(0,1)$ be a symmetric function space in the sense of Lindenstrauss and Tzafriri. 
\begin{enumerate}
\item If $E(0,1)$ is an interpolation space for the couple $(L_1,L_q)$ for some $q<\infty,$ then $E(0,1)$ contains a subspace isometric to a symmetric function space $Z_E^2(0,\infty)$;
\item If $E(0,1)$ is an interpolation space for the couple $(L_p,L_q)$ for some $1<p<q<\infty,$ then the spaces $E(0,1)$ and $Z_E^2(0,\infty)$ are isomorphic.
\end{enumerate}
\end{thm}

Theorem \ref{jmst} above  was recast by Astashkin (and the second-named author) in terms of the Kruglov property (or the Kruglov operator) \cite{AS04, AS05, AS08, AS10, Astashkin},   where the reader  will find  results which cover and extend Theorem \ref{jmst}.   Roughly speaking, a symmetric function  space $E(0,1)$ has the Kruglov property if it is situated sufficiently ``far'' from the space $L_\infty (0,1)$. 
In particular, if $E(0,1)$ contains some $L_p(0,1)$ with $p < \infty$ or, all the more, if its lower Boyd index $\alpha_E > 0$, then  $E(0,1)$ possesses the Kruglov property\cite{Astashkin,AS05,Braverman}.

The main theorem of the present paper establishes a noncommutative version of Theorem \ref{jmst} above.
In the special case  when $\cM=L_\infty(0,\infty)$, we have 
$\cN=L_\infty(0,1)$ and our result specializes to the main result from   \cite{AS10, Astashkin} (and extends Theorem  \ref{jmst}). For the definition of the Kruglov operator $K^{AS}$ see Section 2.3 below (in particular, see Lemma \ref{k vs kas lemma}).

\begin{thm} \label{main}Let $E(0,1)$ be a symmetric function space in the sense of Lindenstrauss and Tzafriri.  
Suppose that the Kruglov operator 
$K^{AS}$ is bounded   on both $E(0,1)$ and its  K\"othe dual $E(0,1)^{\times}.$
\begin{enumerate}
\item For every semifinite von Neumann algebra $(\cM,\tau)$ on a separable Hilbert space $H$, there exists a finite von Neumann algebra $(\cN,\nu)$ such that $Z_E^2(\cM,\tau)$ embeds into $E(\cN,\nu)$ as a complemented subspace. 
\item If, in particular, $\mathcal{M}=\mathcal{R}\bar{\otimes} \mathcal{L}(H),$ then we can take $\mathcal{N}=\mathcal{R}.$ In this case, $Z_E^2(\mathcal{R}\bar{\otimes} \mathcal{L}(H))$ is isomorphic to $E(\mathcal{R}).$
\end{enumerate}
\end{thm}

In the non-commutative setting, we do not have a direct analogue of the Poisson process. However, a suitable non-commutative analogue of the operator $K$ was constructed in \cite{JSZ}, which is our main tool. The classical origin of this notion goes back to the  probability theory and results of Kruglov \cite{Kruglov}. 
The term ``Kruglov property'' and thorough discussion of that property and its applications in Banach space geometry is due to Braverman~\cite{Braverman}.  
In his work, Braverman linked Kruglov property with a remarkable inequality due to Rosenthal \cite{R} and its  significant generalization to the class of symmetric function spaces due to  
Johnson and Schechtman \cite{JS} and which may be briefly characterized as a ``disjointification principle for independent random variables''. 
Braverman's approach was recast in terms of Kruglov operator in already cited papers
\cite{AS04,AS05,AS10}. For an explanation of the deep connection between Kruglov operator and Poisson stochastic integration techniques used in \cite{JMST} and \cite{LT2},
 we refer the reader to \cite[Section 9]{AS10}.


We emphasize that while the notion of Kruglov operator is the fundamental tool underlying this paper as well as results in  \cite{JMST,LT2} and \cite {AS08, AS10, Astashkin} our ``noncommutative''  proofs are totally different from those employed in all just cited ``commutative'' papers treating Question \ref{question1}.


 \section{Preliminaries}
 \subsection{Symmetric function spaces}

Let $I$ be $(0,1)$ or $(0,\infty)$ and let $L(I)$ denote the space of all  (real)  Lebesgue-measurable functions $x$ on $I.$
For a Lebesgue-measurable function $x$ on $I,$ we define its {\it distribution function} by the formula
$$d_x(s)=m(\{t:\ x(t)>s\}),\quad s\in\mathbb{R},$$
where $m$ stands for Lebesgue measure. 
Denote by $S(I)$ the subalgebra of $L(I)$
 consisting of all functions $x$ such that $d_{|x|}(s) < \infty$  for some 
 $s > 0$.
 
Two measurable functions $x$ and $y$ are called {\it equimeasurable} (written, $x\sim y$) if the  distribution functions $d_{x_+}$ and $d_{y_+}$ (respectively, 
$d_{x_-}$ and $d_{y_-}$) coincide. In particular, for every measurable function $x\in S    (I),$ the function $|x|$ is equimeasurable with its {\it decreasing rearrangement} $\mu(x)$ defined by the formula
$$\mu(t;x):=\inf \{\tau\geq0:\ d_{|x|}(\tau)<t \},\quad t>0.$$
If $ x,y\in S(I),$ then $\mu(x)=\mu(y)$ if and only if $|x|$ and $|y|$ are equimeasurable.

\begin{defi} [see e.g. \cite{KPS,LT2}] Let $X\subset S(I)$ be a Banach space.
\begin{enumerate}
\item $X$ is said to be a Banach function space if, from $x\in X,$ $y\in S(I)$ and $|y|\leq |x|,$ it follows that $y\in X$ and $\left\|y\right\|_X\leq \left\|x\right\|_X.$
\item a Banach function space $X$ is said to be a symmetric if, for every $x\in X$ and any measurable function $y,$ the assumption $\mu(y)=\mu(x)$ implies that $y\in X$ and $\left\|y\right\|_X=\left\|x\right\|_X.$
\end{enumerate}
\end{defi}
A symmetric function space $X(0,1)$ is said to be 
a symmetric function space  {\it in the sense of Lindenstrauss and Tzafriri} { (see \cite[p.118]{LT2})} if   $X(0,1)$   is separable or has the Fatou property (i.e., for every upwards directed net $\{ a_\beta\}$ in $X(0,1)_+$ with $\sup_\beta\norm{a_\beta}_X<\infty $, there exists $a\in X(0,1)_+$ such that $a_\beta\uparrow a\in X(0,1)$ and $\norm{a}_X=\sup _\beta\norm{a_\beta}_X$)\cite{LT2}. 
In what follows,  we always assume that  $X(0,1)$   is separable or has the Fatou property. 

Without lost of generality, in what follows we may always assume that \begin{align}
\label{normalized}
\left\|\chi_{_{(0,1)}}\right\|_X=1
\end{align} for any symmetric function space  $X=X(0,1)$ on $(0,1)$.

\subsection{Noncommutative symmetric spaces}For a  detailed exposition of material in this subsection,  we refer to \cite{DP2,DPS,KS,LSZ}.
In what follows,  $  H  $ is a complex separable  infinite-dimensional Hilbert space and $\mathcal{L}(  H  )$ is the
$*$-algebra of all bounded linear operators on $  H  $ equipped with the uniform norm $\left\|\cdot\right\|_\infty$, and
$\mathbf{1}$ is the identity operator on $  H  $.
Let $\mathcal{M}$ be
a von Neumann algebra on $  H  $.
We denote by $\cP(\cM)$ the collection  of all projections in $\cM$, by $\cM'$ the commutant of $\cM$ and by $\cZ(\cM)$ the center of $\cM$.
For more information about von Neumann algebras, see e.g. \cite{KR,Tak,Dixmier}.

A closed, densely defined operator $x:\mathfrak{D}\left( x\right) \rightarrow   H   $ with the domain $\mathfrak{D}\left( x\right) $ is said to be {\it affiliated} with $\mathcal{M}$
if $yx\subseteq xy$ for all $y\in \mathcal{M}^{\prime }$, where $\mathcal{M}^{\prime }$ is the commutant of $\mathcal{M}$.
A  closed,
densely defined
operator $x:\mathfrak{D}\left( x\right) \rightarrow   H   $ affiliated with $\cM $ is said to be
{\it measurable}  if  there exists a
sequence $\left\{ p_n\right\}_{n=1}^{\infty}\subset \cP\left(\mathcal{M}\right)$, such
that $p_n\uparrow \mathbf{1}$, $p_n(  H  )\subseteq\mathfrak{D}\left(x\right) $
and $\mathbf{1}-p_n$ is a finite projection (with respect to $\mathcal{M}$)
for all $n$.
 The collection of all measurable
operators with respect to $\mathcal{M}$ is denoted by $S\left(
\mathcal{M} \right) $, which is a unital $\ast $-algebra
with respect to strong sums and products (denoted simply by $x+y$ and $xy$ for all $x,y\in S\left( \mathcal{M%
}\right) $).

From now on, let $\mathcal{M}$ be a
semifinite von Neumann algebra equipped with a faithful normal
semifinite trace $\tau$.

An operator $x\in S\left( \mathcal{M}\right) $ is called $\tau$-measurable if
$\tau(e^{|x|}(s,\infty))<\infty$ for sufficiently large $s$, where
by $e^{|x|}$ is denoted the spectral measure of $|x|$.
The collection $S\left( \mathcal{M}, \tau\right)
$ of all $\tau $-measurable
operators is a unital $\ast $-subalgebra of $S\left(
\mathcal{M}\right) $.
The algebra $S_0(\cM,\tau)$ is the subalgebra of $S(\cM,\tau)$ consisting of all $\tau$-measurable operators $x$ with $\tau(e^{|x|(s,\infty )})<\infty $
for all $s>0$. 

Consider the algebra $\mathcal{M}=L^\infty(0,\infty)$ of all
Lebesgue measurable essentially bounded functions on $(0,\infty)$.
The algebra $\mathcal{M}$ can be seen as an abelian von Neumann
algebra acting via multiplication on the Hilbert space
$\mathcal{H}=L^2(0,\infty)$, with the trace given by integration
with respect to Lebesgue measure $m.$
It is easy to see that the
algebra of all $\tau$-measurable operators
affiliated with $\mathcal{M}$ can be identified with
the algebra $S(0,\infty)$.

\begin{definition}\label{mu}\cite{Fack,Nelson,Se}
Let $x\in
S(\mathcal{M},\tau)$. The generalized singular value function $\mu(x):t\mapsto  \mu(t;x)$ of
the operator $x$ is defined by 
$$
\mu(s;x)
=
\inf \left\{
 s:d_{|x|}(s)\le t\right \}, ~t\ge 0,
$$
where $d_b(s) :=\tau(e^b(s,\infty))$, $s\in \mathbb{R}$, $b=b^*\in S(\cM,\tau)$. 
\end{definition}

Operators $x=x^*\in S(\cM,\tau)$ and $y=y^*\in S(\cN,\nu)$ are  
  called {\it equimeasurable} (written, $x\sim y$) if the  distribution functions $d_{x_+}$ and $d_{y_+}$ (respectively, 
$d_{x_-}$ and $d_{y_-}$) coincide.


\begin{definition}\label{def:symmetric}\cite{DPS,KS,LSZ}
 A linear subspace $E$ of $S(\cM,\tau)$ equipped with a complete norm $\norm{\cdot}_E$, is called a   symmetric space (of $\tau$-measurable operators) if $x\in S(\cM,\tau)$, $y \in E$ and $\mu(x)\le \mu(y)$ imply that $x\in E$ and $\norm{x}_E \le \norm{y}_E$.
\end{definition}

It is well-known that any symmetric space $E$ is a normed $\cM$-bimodule, that is, $axb\in E$ for any $x\in E$, $a,b\in \cM$ and $$\left\|axb\right\|_E\leq \left\|a\right\|_\infty\left\|b\right\|_\infty \left\|x\right\|_E$$
see e.g. \cite{DP2,DPS}.
%

%
%
%
%
%

 A wide class of  symmetric operator spaces associated with the von Neumman algebra $\cM$ can be constructed from concrete symmetric function spaces studied extensively in e.g. \cite{KPS,Bennett_S,LT2}. Let 
 $\cE :=  E(0,\infty) $ be a   symmetric function space on the semi-axis $(0,\infty)$ (or $\cE : =E(0,1) $ for the case when  $(\cM,\tau)$ is \emph{a noncommutative probability space}, i.e., $\tau$ is a faithful normal tracial state). Then the pair 
 $$E(\cM,\tau)=\{x\in S(\cM,\tau):\mu(x)\in \cE \},\quad \left\|x\right\|_{E(\cM,\tau)}:=\left\|\mu(x)\right\|_{\cE }$$ is a  symmetric space on $\cM$ \cite{KS} (see also \cite{LSZ}). 
%
%

We write $x\prec\prec y$ (and say that $x$ is submajorized by $y$ in the sense of Hardy--Littlewood--P\'{o}lya) if
\begin{align}\label{HLP}
\int_0^t\mu(s;x)ds\leq\int_0^t\mu(s;y)ds,\quad t>0.
\end{align}
If the norm $\norm{\cdot}_E$ of a symmetric space $E(\cM,\tau)$ is monotone with respect to the Hardy--Littlewood--P\'{o}lya submajorisation (that is, if $x,y\in E(\cM,\tau)$ with $x\prec\prec y$, then $\norm{x}_E\le \norm{y}$), then $E(\cM,\tau)$ is called a strongly symmetric space.
Symmetric function spaces   in the sense of Lindenstrauss and Tzafriri  are necessarily strongly symmetric \cite{DPS,LT2}.



 The so-called K\"{o}the dual is identified with an important part of the dual space.
If $E (\cM,\tau) \subset S(\cM,\tau)$ is a symmetric space, then the K\"{o}the dual $E(\cM,\tau)^\times $ of $E(\cM,\tau)$ is defined by setting\cite{DPS}
$$ E(\cM,\tau)^\times =\left\{   x\in S(\cM,\tau) : \sup_{\|y\|_E\le 1, y\in E(\cM,\tau)}\tau (|xy|)   <\infty    \right\}.$$

Let $E(\cM,\tau)$ be a strongly symmetric space. 
If $y \in E(\cM,\tau)^\times$, then the linear functional $$\phi_y:x \mapsto \tau(xy) , ~ x \in E(\cM,\tau),$$ is continuous on $E(\cM,\tau)$, see \cite[Proposition 22 (i)]{DP2}. In addition, $\norm{\phi_y}_{E(\cM,\tau)^*}=\norm{y}_{E(\cM,\tau)^\times}$, where  $E(\cM,\tau)^*$ is the dual of  $E(\cM,\tau)$ (see, e.g. \cite{DP2,DPS}). The K\"{o}the dual $E(\cM,\tau)^\times  $ can be identified as   a subspace of the Banach dual $E(\cM,\tau)$ via the trace duality~\cite[p.228]{DP2}. If $y \in E(\cM,\tau)^\times$, then the functional $ \phi_y$ is normal, that is, $ x_\alpha \downarrow 0 $ in $ E(\cM,\tau)$ implies that $ \phi_y(x_\alpha) \rightarrow_\alpha 0 $\cite[Definition 4.2.16]{DPS}. Note that every normal functional $ \phi \in E(\cM,\tau)^*$ is of the form $\phi_y$ for some $ y \in E(\cM,\tau)^\times $ (see e.g. \cite[Theorem 37]{DP2}). Since $E(\cM,\tau)^ \times$ separates the point of $E(\cM,\tau)$ (see \cite[Corollary 4.3.9]{DPS}), it follows that  the weak topology $\sigma(E(\cM,\tau),E(\cM,\tau)^\times)$ is a Hausdorff topology.  
For brevity, we write $\sigma(E,E^\times)$ instead of $\sigma(E(\cM,\tau),E(\cM,\tau)^\times)$.


Recall that $$x\in (L_1+L_\infty)(\cM,\tau) :=\{a\in S(\cM,\tau):\mu(a)\in L_1(0,\infty)+L_\infty(0,\infty)\}$$ can be equipped with a norm $\norm{x}_{L_1+L_\infty}= \int_0^1 \mu(s;x)ds $
and $$x\in (L_1\cap L_\infty)(\cM,\tau) :=\{a\in S(\cM,\tau):\mu(a)\in L_1(0,\infty)\cap L_\infty(0,\infty)\}$$ can be equipped with a norm $\norm{x}_{L_1\cap L_\infty}:=\max \{\norm{x}_1,\norm{x}_\infty\}$.
In particular, we have $$\left(
 L_1\cap L_\infty\right) (\cM,\tau)^\times
 = (L_1+L_\infty)(\cM,\tau)  $$ and $$\left(
 L_1\cap L_\infty\right) (\cM,\tau)  =(L_1+L_\infty)(\cM,\tau) ^\times,$$
 see e.g. \cite[Example 4]{DP2} and \cite[Example 4.3.13]{DPS}.

 \subsection{The Kruglov operator}\label{S:k}
 {
 We denote by $\cR$ 
 the hyperfinite $II_1$-factor and by $\cR\bar{\otimes} \mathcal{L}(H)$ the hyperfinite $II_\infty$ factor, see e.g. \cite{KR,Tak}.}
 Let $\Omega =\Pi_{k=0}^\infty (0,1)$ be the (infinite dimensional) hypercube equipped with the product Lebesgue measure 
$dm^\infty $. 
The Kruglov operator $K^{AS}$ acts from $L_1(0,1)$ to $L_1(\Omega)$ by the following formula
$$K^{AS}x=\sum_{k=1}^\infty \sum_{m=1}^k \chi_{_{A_k}} \otimes \chi_{(0,1)}^{\otimes (m-1)} \otimes x \otimes \chi_{(0,1)}^{\otimes \infty 
}, ~x\in L_1(0,1), $$
where $A_k$, $k\ge 0$, are pairwise disjoint sets with $m(A_k) = \frac{1}{e\cdot k!} $ for all $k\ge 0$ (so that
 $\cup_{k\ge 0}
A_k =(0,1)$) and where $\chi_{_B}$ is the indicator function of the measurable set $B\subset (0,1)$\cite{AS04,AS05,AS10}. 

Consider the Poisson process (see e.g. \cite{BulSh} and \cite[p.204]{LT2}). This is the mapping $s\to P_s$ from $(0,\infty)$ to $L_1(0,1)$ such that 
\begin{enumerate}
\item $P_0=0;$
\item $P_s-P_t$ is a Poisson random variable with parameter $s-t$ for every $s>t\ge 0;$
\item if $(s_k)_{k\geq0}$ is an increasing sequence, then the random variables $(P_{s_{k+1}}-P_{s_k})_{k\geq0}$ are independent;
\end{enumerate}
We now introduce the integration operator $K$ with respect to the Poisson process. The domain of $K$ is, by definition, $L_1(0,\infty)$ and the codomain is, by definition, $L_1(0,1).$ The operator $K$ is defined by the formula
$$Kf=\int_0^{\infty}f(s)dP_s,\quad f\in L_1(0,\infty).$$
That $K$ indeed sends $L_1(0,\infty)$ to $L_1(0,1)$ is clearly seen from the following computation:
$$\int_0^1 (Kf)(s) ds =\int_0^{\infty}f(s)d\Big(\int_0^1P_s(t)dt \Big)=\int_0^{\infty}f(s)ds.$$
The following formula is a crucial feature of the operartor $K$ which delivers a useful expression\footnote{The equality \eqref{char for comm K} is folklore, but we were unable to find a concrete reference. For this reason, we sketch a proof.

Recall that functions of the shape $\sum_{k=0}^na_k\chi_{_{(k\epsilon,(k+1)\epsilon)}},$ $n\in\mathbb{N},$ $\epsilon>0,$  are dense in $L_1(0,\infty).$ As $K:L_1(0,\infty)\to L_1(0,1)$ is bounded, it suffices to prove the equality \eqref{char for comm K} only for $f=\sum_{k=0}^na_k\chi_{_{(k\epsilon,(k+1)\epsilon)}}.$ In this case, we have 
$$Kf=\sum_{k=0}^na_kQ _k,\quad Q_k=K\left(\chi_{_{(k\epsilon,(k+1)\epsilon)}}\right).$$
By the definition of the Poisson process, $(Q_k)_{k=0}^n$ is a sequence of independent Poisson random variables with parameter $\epsilon.$ 
Thus,
\begin{align*}
\int_0^1e^{i(Kf)(s)} ds&=
\prod_{k=0}^n\int_0^1e^{ia_k Q_k (s)}ds =
\prod_{k=0}^n \left (  \sum_{m=0} e^{ia_k m  }  \frac{\epsilon^m }{m ! e^\epsilon}  \right) 
\\
&=\prod_{k=0}^n\exp\left(\epsilon(e^{ia_k}-1)\right) =\exp\left (\epsilon\sum_{k=0}^n(e^{ia_k}-1) \right) =\exp\left (\int_0^{\infty}(e^{if(s) }-1)ds \right) .
\end{align*}
} for the characteristic function of the random variable $Kf,$ $f\in L_1(0,\infty)$:
\begin{equation}\label{char for comm K}
\int_0^1e^{i(Kf)(s)}ds =\exp\Big(\int_0^{\infty}(e^{if(s)}-1)ds \Big).
\end{equation}

The relation between the operators $K$ and $K^{AS}$ was indicated on p.1071 in \cite{AS10}.
\begin{lem}\label{k vs kas lemma} For every $f\in L_1(0,1),$ the random variables $Kf$ and $K^{AS}f$  are equimeasurable.
\end{lem}
\begin{proof} If $f\in L_1(0,1),$ then it is easy to see that \eqref{char for comm K} also holds for $K^{AS}f$ (integration over $(0,1)$ in the right hand side should be replaced with the integration over $\Omega$),
 see e.g. \cite[Eq.(34)]{AS10} and \cite[Eq.(60)]{JSZ}. Hence, characteristic functions of $Kf$ and $K^{AS}f$ coincide. Thus, so are their distribution functions.
\end{proof}

Recall that a symmetric function space $E(0,1)$ has {\it the Kruglov property}  if $K^{AS}$ is bounded from $E(0,1)$ into $E(\Omega)$. 

In the non-commutative setting, we do not have a direct analogue of the Poisson process. However, a suitable non-commutative analogue of the operator $K$ (and $K^{AS}$) was constructed in \cite{JSZ}.

Let $\cM$ be a semifinite von Neumann algebra (on a separable Hilbert space $  H  $)  equipped with a faithful normal trace $\tau$. 
The Kruglov operator $\mathscr{K}_{_\cM} $ is a \emph{positive} bounded  operator from  $L_1(\cM,\tau)$ into $L_1\left(
\cN_{_\cM},\nu_{_{\cN_{\cM}}}\right)$ for some finite von Neumann algebra $\cN_{_\cM} $ equipped with a faithful normal tracial state $\nu_{_{\cN_{\cM}}}$ such that \cite[Theorem 32]{JSZ}
\begin{equation}\label{char for noncomm K}
\nu_{_\cM}  ({\rm exp}(i \mathscr{K}_{_\cM}x)) ={\rm exp}(\tau({\rm exp}(ix)-1)),\quad x=x^{\ast} \in L_1(\cM,\tau).
\end{equation}
Moreover, if $\cM$ is hyperfinite, then $\cN_{_\cM}$ can be chosen to be hyperfinite\cite[Theorem 32]{JSZ}. 

The operator $\mathscr{K}_{_\cM}$ possesses the following properties:

\begin{lem}\label{equimeasurability lemma} For every element $x=x^{\ast}\in L_1(\cM,\tau)$ equimeasurable with a real-valued $y\in L_1(0,\infty),$ the operator $\mathscr{K}_{_\cM}x \in L_1\left(\cN_{_\cM},\nu_{_{\cN_{\cM}}} \right)$ and the function $Ky\in L_1(0,1)$ are equimeasurable.
\end{lem}
\begin{proof} Indeed, we have
$$\nu_{_{\cN_\cM}} ({\rm exp}(it\mathscr{K}_{_\cM}x))\stackrel{\eqref{char for noncomm K}}{=}{\rm exp}(\tau({\rm exp}(itx)-1))$$
and
$$\int_0^1e^{it(Ky)(s)}ds\stackrel{\eqref{char for comm K}}{=}\exp\left(
\int_0^{\infty}(e^{ity(s)}-1)ds
\right )  .$$
Since $x$ and $y$ are equimeasurable, it follows that
$$\tau({\rm exp}(itx)-1)=\int_0^{\infty}(e^{ity(s)}-1) ds $$
and, hence,
$$\nu_{_{\cN_{\cM}}} ({\rm exp}(it\mathscr{K}_{_\cM}x))=\int_0^1e^{it(Ky)(s)}ds .$$
Therefore, the characteristic functions of the random variables $\mathscr{K}_{_\cM}x$ and $Ky$ coincide. Hence, those random variables are equimeasurable \cite[Corollary of Theorem~6.3.2, p.141]{Borovkov} (see also \cite[Appendix A]{Astashkinbook}).
\end{proof}

\begin{lem}\cite[Corollary 33(i)]{JSZ} If self-adjoint operators $x,y\in L_1(\cM,\tau)$ are equimeasurable, then $ \mathscr{K}_{_\cM}x, \mathscr{K}_{_\cM}y \in L_1\left(\cN_{\cM},\nu_{_{\cN_{\cM}}}\right)$ are equimeasurable too.
\end{lem}

The next result was obtained in  \cite[Corollary 33(ii)]{JSZ}.

\begin{lem} The operator $\mathscr{K}_{_\cM} $ acts boundedly from $(L_1\cap L_2)(\cM,\tau)$   into $L_2\left(\cN_{_\cM},\nu_{_{\cN_{_\cM}}}\right).$  
\end{lem}

If, in particular, $\cM=\cR\bar{\otimes} B(  H  )$, then $\mathscr{K}_{_\cM} $ acts from $L_1(\cR\bar{\otimes} B(  H  ), \nu\otimes {\rm Tr} )$ into $L_1(\cR,\nu)$~\cite[Corollary 34]{JSZ}, where $\nu$ stands for the   faithful normal tracial state on $\cR.$  
  

\section{main results}
Recall the following well-known fact in the theory of Banach spaces.

\begin{lem}\label{standard lemma} Let $X_0$ and $X_1$ be two Banach spaces and let 
$$T_0:X_0\to X_1 \mbox{ and } T_1:X_1\to X_0$$
 be two bounded linear operators. If $T_1\circ T_0={\rm id},$ then $T_0(X_0)$ is a complemented subspace in $X_1.$ In particular, $X_0$ is isomorphic to a complemented subspace in $X_1.$
\end{lem}

Let $\cN$ be a finite von Neumann algebra equipped with a faithful normal tracial state $\nu$. 
Self-adjoint operators affiliated with $\cN$ are called noncommutative random variables. 


The following lemma and Corollary \ref{second cumulant lemma} below were  established in \cite{JSZ} {for positive operators}, see e.g. \cite[Proof of Corollary 33(ii)]{JSZ}.  
{The same argument yields the case for self-adjoint operators. For the sake of completeness, we present a proof below.}

\begin{lem}\label{second cumulant commutative lemma} If $f\in (L_1\cap L_2)(0,\infty)$ is real-valued, then
$$\int_0^1((Kf)(t))^2 dt =\int_0^{\infty}f(t)^2dt+\left(\int_0^{\infty}f(t)dt\right)^2,$$
where $K$ is the integration operator with respect to the Poisson process introduced in Section \ref{S:k}. 
\end{lem}
\begin{proof} Set
$$F(t)=\int_0^1e^{it(Kf)(s)} ds ,\quad t\in\mathbb{R}.$$
By the Dominated Covergence Theorem, we have
$$F'(0)=\lim_{t\to0}\frac{F(t)-F(0)}{t}=\lim_{t\to0}\int_0^1\frac1t \left(e^{it(Kf)(s)}-1\right )ds =i\int_0^1 (K f)(s)ds $$
and 
\begin{eqnarray}
\label{F''}
\begin{split}
F''(0)&=2\lim_{t\to0}\frac{F(t)-F(0)-F'(0)t}{t^2}\\
&=\lim_{t\to0}\int_0^1\frac2{t^2}\left(e^{it(Kf)(s)}-1-it (Kf)(s)\right)ds \\
&=-\int_0^1((Kf)(s))^2ds .
\end{split}
\end{eqnarray}

Taking into account that (see \eqref{char for comm K})
$$F(t)=\exp(G(t))$$
where $ G(t)=\int_0^{\infty}(e^{itf(s)}-1) ds ,~ t\in\mathbb{R},$
and using the chain rule, we write
\begin{align*}
F'(0)=\exp(G(0))\cdot G'(0)=G'(0)&=\lim_{t\to0}\frac{G(t)}{t}\\
&=\lim_{t\to0}\int_0^{\infty}\frac1t \left( e^{itf(s)}-1\right)ds =i\int_0^{\infty}f(s)ds ,
\end{align*}
and 
\begin{align*}
F''(0)&=\exp(G(0))\cdot (G'(0))^2+\exp(G(0))\cdot G''(0)\\
& =(G'(0))^2+G''(0) \\
&=-\left(\int_0^{\infty}f(s)ds\right)^2+2\lim_{t\to0}\frac{G(t)-G(0)-tG'(0)}{t^2}\\
&=-\left (\int_0^{\infty}f(s)ds\right)^2+\lim_{t\to0}\int_0^{\infty}\frac2{t^2}\left(
e^{itf(s)}-1-itf(s) \right)ds \\
&=-\left(\int_0^{\infty}f(s)ds\right )^2-\int_0^{\infty}f(s)^2ds,
\end{align*}
which together with \eqref{F''} completes the proof.
\end{proof}

\begin{cor}\label{second cumulant lemma} Let $\cM$ be a semifinite von Neumann algebra equipped with a semifinite faithful normal trace $\tau$.  If $x=x^*\in (L_1\cap L_2)(\cM,\tau),$ then
$$\nu_{_{\cN_{\cM}}}  \left((\mathscr{K}_{_\cM}x)^2\right)= \tau \left(x^2\right)+ \tau (x) ^2.$$
In particular, if $\cM=\cR\bar{\otimes }\mathcal{L}(H)$, then  
$$\nu \left( (\mathscr{K}x)^2\right)=\left(\nu \otimes{\rm Tr}\right)\left(x^2\right)+\big((\nu\otimes{\rm Tr})(x)\big)^2.$$
\end{cor}
\begin{proof} Let  $f\in (L_1\cap L_2)(0,\infty)$ be a real-valued function which is equidistributed with $x.$ By Lemma~\ref{equimeasurability lemma},  $\mathscr{K}_{_\cM}x$ is equidistributed with $Kf.$ The assertion follows now from Lemma~\ref{second cumulant commutative lemma}.
\end{proof}

Let
$$\iota:L_{\infty}(0,1)\bar{\otimes}\mathcal{R}\to\mathcal{R}$$
be any fixed trace-preserving $\ast$-homomorphism. Define a  trace-preserving $\ast$-homomorphism $$\theta:L_{\infty}(0,1)\bar{\otimes}\mathcal{R}\bar{\otimes}\mathcal{L}(H)\to\mathcal{R}\bar{\otimes}\mathcal{L}(H)$$ by setting 
$$\theta=\iota\otimes{\rm id}_{_{\mathcal{L}(H)}}.$$ 
Abusing the notation, we denote the mapping
$$\mathscr{K}\circ\theta:L_1(L_{\infty}(0,1)\bar{\otimes}\mathcal{R}\bar{\otimes}\mathcal{L}(H))\to L_1(\mathcal{R})$$ also by $\mathscr{K}.$
 
 Let $\cM$ be a semifinite von Neumann algebra equipped with a semifinite faithful normal trace $\tau$.
We also need the map $\mathscr{R}_{_\cM}:L_1(\cM,\tau)\to L_1(\cN_{_{L_\infty (0,1) \bar{\otimes } \cM  } })$ defined by the formula
\begin{align}\label{defRM}
\mathscr{R}_{_\cM}:x\to\mathscr{K}_{_{L_\infty(0,1) \bar{\otimes} \cM }}(r\otimes x),\quad 
x\in 
L_1(\cM,\tau),
\end{align}
and  the map $\mathscr{R}:L_1(\mathcal{R}\bar{\otimes}\mathcal{L}(H))\to L_1(\cR)$ defined by the formula
\begin{align}\label{defR}
\mathscr{R}:x\to\mathscr{K}(r\otimes x),\quad x\in L_1(\mathcal{R}\bar{\otimes}\mathcal{L}(H)),
\end{align}
where $r$ stands for the Rademacher function $r_1$ \cite{LT1}. 
\begin{lem} \label{hilbertian}

Let $\cM$ be a semifinite von Neumann algebra equipped with a semifinite faithful normal trace $\tau$.\begin{enumerate}
\item Let $x,y\in\cM.$ For every $x,y\in (L_1\cap L_2)(\cM,\tau),$ we have
$$\nu_{_{\cN _{L_\infty (0,1) \bar{\otimes }\cM } } } \left(
\mathscr{R}_{_\cM}(x)\cdot\mathscr{R}_{_\cM} (y)\right)=\tau(xy).$$
Consequently, $\mathscr{R}_{_\cM} $ extends to an isometry $\mathscr{R}_{_\cM} 
:L_2(\cM,\tau)\to L_2 (\cN_{_{L_\infty (0,1) \bar{\otimes } \cM  } })$ and the above equality holds for every $x,y\in L_2(\cM,\tau).$
\item 
Let $x,y\in \mathcal{R}\bar{\otimes}\mathcal{L}(H).$ For every $x,y\in (L_1\cap L_2)(\mathcal{R}\bar{\otimes}\mathcal{L}(H)),$ we have
$$\tau (\mathscr{R}(x)\cdot\mathscr{R}(y))=(\tau\otimes{\rm Tr})(xy).$$
Consequently, $\mathscr{R}$ extends to an isometry $\mathscr{R}:L_2(\mathcal{R}\bar{\otimes}\mathcal{L}(H))\to L_2(\mathcal{R})$ and the above equality holds for every $x,y\in L_2(\mathcal{R}\bar{\otimes}\mathcal{L}(H)).$
\end{enumerate}
\end{lem}
\begin{proof} 
It suffices to prove (1). 
For any $x,y\in (L_1\cap L_2)(\cM,\tau)$, we have 
\begin{align*}
&~\quad  \nu_{_{\cN _{L_\infty (0,1) \bar{\otimes }\cM } } } \left(
\mathscr{R}_{_\cM}(x)\cdot\mathscr{R}_{_\cM} (y)\right)\\
&=\nu_{_{\cN _{L_\infty (0,1) \bar{\otimes }\cM } } } \left(
\mathscr{R}_{_\cM}(x_1+ix_2 )\cdot\mathscr{R}_{_\cM} (y_1+iy_2)\right)\\
&=\nu_{_{\cN _{L_\infty (0,1) \bar{\otimes }\cM } } } \Big(
\mathscr{R}_{_\cM}(x_1  )\cdot\mathscr{R}_{_\cM} (y_1 )+i
\mathscr{R}_{_\cM}(x_1  )\cdot\mathscr{R}_{_\cM} (y_2 )\\
&\qquad\qquad \qquad  \qquad \quad \quad  +i
\mathscr{R}_{_\cM}(x_2  )\cdot\mathscr{R}_{_\cM} (y_1 )-
\mathscr{R}_{_\cM}(x_2  )\cdot\mathscr{R}_{_\cM} (y_2 )
\Big)
\end{align*}
and 
\begin{align*}
\nu_{_{\cN _{L_\infty (0,1) \bar{\otimes }\cM } } } \left(xy \right)
&=\nu_{_{\cN _{ L_\infty (0,1)}  \bar{\otimes }\cM } } \left(
 \left(x_1+ix_2 \right)\cdot  \left(y_1+iy_2\right)\right)\\
&=\nu_{_{\cN _{L_\infty (0,1)} \bar{\otimes }\cM  } } \Big(
 x_1  \cdot  y_1 +i
 x_1  \cdot  y_2 +i x_2 
 \cdot   y_1 -
 x_2  \cdot   y_2 
\Big),
\end{align*}
where $x=x_1+ix_2$, $y=y_1+iy_2$, $x_i=x_i^*$ and $y_i=y_i^*$ for $i=1,2$. 
It follows from the above equalities that it suffices to prove the statement assuming, in addition, that the operators $x$ and $y$ are self-adjoint.

By Corollary \ref{second cumulant lemma}, we have
$$\nu_{_{\cN_{\cM \bar{\otimes}  L_\infty (0,1) }}}\left(
(\mathscr{K}_{_{L_\infty (0,1)\bar{\otimes }\cM } }z)^2
\right)=\left(
\int\otimes\tau \right) \left( z^2 \right)+\left(  \left (\int\otimes\tau \right )(z)\right)^2$$
for every $z=z^*\in(L_1\cap L_2)\left(
L_{\infty}(0,1)\bar{\otimes}\cM \right).$ 

If $x=x^*,y=y^*\in (L_1\cap L_2)(\cM,\tau ),$ then, setting $$z=r\otimes x, ~z=r\otimes y \mbox{ and } z=r\otimes (x+y)$$ respectively,  we obtain
$$\nu_{_{\cN_{\cM \bar{\otimes}  L_\infty (0,1) }}}\left ( (\mathscr{R}_{_\cM}x)^2\right)=\left(\int\otimes \tau\right)\left( \chi_{_{(0,1)}}\otimes x^2\right)=\tau(x^2), $$
$$ \nu_{_{\cN_{\cM \bar{\otimes}  L_\infty (0,1) }}}   \left( (\mathscr{R}_{_\cM}y)^2\right)=\left(\int\otimes \tau\right)\left(\chi_{_{(0,1)}}\otimes  y^2\right)=\tau(y^2)$$
and 
$$\nu_{_{\cN_{\cM \bar{\otimes}  L_\infty (0,1) }}}\left ((\mathscr{R}_{_\cM}(x+y))^2\right)=\left(\int\otimes \tau\right) \left ( \chi_{_{(0,1)}}\otimes(x+y)^2\right )=\tau((x+y)^2).$$
Subtracting the first two equalities from the last one, we write
$$ \nu_{_{\cN_{\cM \bar{\otimes}  L_\infty (0,1) }}} \left( \mathscr{R}_{_\cM}(x  ) \mathscr{R}_{_\cM}(y ) +\mathscr{R}_{_\cM}(y ) \mathscr{R}_{_\cM}( x ) \right )=   \tau  (xy+yx).$$
Using trace property, we obtain
$$\nu_{_{\cN_{\cM \bar{\otimes}  L_\infty (0,1) }}} \left (\mathscr{R}_{_\cM}x\cdot \mathscr{R}_{_\cM}y\right)=  \tau   (xy),\quad x,y\in (L_1\cap L_2)(\cM,\tau ).$$
This completes the proof. 
\end{proof}

The following lemma was known in the setting of symmetric function spaces having order continuous norm or the Fatou property
\cite[Proposition 11]{Braverman}.
 
\begin{lem}\label{braverman lemma} Let $E(0,1)$ be a strongly symmetric function space. If $y_1,y_2\in E(0,1)$ are real-valued, independent and mean zero, then
$$\left\|y_1+y_2\right\|_E\geq \max\left\{\left\|y_1\right\|_E,\left\|y_2\right\|_E\right\}.$$ 
\end{lem}
\begin{proof} Consider the function $y_1\otimes 1+1\otimes y_2$ on $(0,1)^2.$  Since $y_2$ is mean zero, it follows that
$$y_1\otimes 1=\left({\rm id}\otimes\int_0^1\right)(y_1\otimes 1+1\otimes y_2).$$
Hence,
$$y_1\otimes 1\prec\prec y_1\otimes 1+1\otimes y_2.$$
Since $y_1$ and $y_2$ are independent, it follows that $y_1+y_2$ is equimeasurable with $y_1\otimes 1+1\otimes y_2.$ Clearly, $y_1$ is equimeasurable with $y_1\otimes 1.$ Thus, $y_1\prec\prec y_1+y_2.$ Since the norm is monotone with respect to the Hardy--Littlewood--P\'{o}lya submajorisation~\eqref{HLP}, it follows that $$\left\|y_1\right\|_E \leq \left\|y_1+y_2\right\|_E.$$ The same argument shows that $\left\|y_2\right\|_E\leq \left\|y_1+y_2\right\|_E.$
This completes the proof.
\end{proof}
 
 Let $x \in S(I)$, $I=(0,\infty)$ or $(0,1)$. 
 If $t>0,$ the dilation operator $\sigma_{t}$ is defined by setting 
 $(\sigma_{t})x(s)=x\left(\frac{s}{t}\right),$ $s>0,$ in the case of the semi-axis\cite{LT2,Bennett_S,KPS}. In the case of the interval $(0,1),$ the operator $\sigma_{t }$ is defined by
 $$
 \sigma_{ t }x(s)=
 \begin{cases}
x\left(\frac{s}{t}\right),& s\leq\min\{1, t \};\\
  0,&  t <s\leq1.
 \end{cases}
  $$

Let $$\psi(t): =\int_0^t \mu\left( s;K  \chi_{(0,1)}\right)ds  ,~t\in (0,1).$$
Recall that  $Kf$ and $K^{AS}f$  are equimeasurable for every $f\in L_1(0,1),$ see Section~\ref{S:k}.
We have   $K  :L_\infty (0,1)\to M_\psi $ and $\norm{K  }_{L_\infty \to M_\psi}=1$ (see \cite[p.265]{ASZ}, which yields the following lemma in the setting of $f\in L_\infty(0,1)$). 
Here, $M_\psi$ stands for the Marcinkiewicz space consisting of all elements $x\in S(0,1)$
  such that \cite{KPS}
$$\norm{x}_{M_\psi}:= \sup_{0<t\le 1} \frac{ \int_0^t \mu(s;x)ds}{ \psi(t)}<\infty. $$

\begin{lem}\label{mpsi estimate} For every symmetrically distributed $f\in L_{\infty}(0,\infty)$  supported on $(0, m)$  for some finite positive number $m$ we have
$$\left\|Kf\right\|_{M_{\psi}}\leq c_{{\rm abs}}\left\|f\right\|_{L_2\cap L_{\infty}},$$
where $K:L_1(0,\infty)\to L_1(0,1)$ is the integration operator with respect to the Poisson process defined in Section \ref{S:k}.
\end{lem}
\begin{proof} Set
$$H_nf=\sum_{k=0}^{n-1}1^{\otimes k}\otimes \sigma_{\frac1n}f\otimes 1^{\otimes\infty},\quad n\geq m.$$	
We claim that the sequence $\{H_nf\}_{n\geq m}$ converges in distribution to $Kf$ as $n\to \infty.$ Indeed, 
\begin{align*}
 \int_{\Omega}e^{it(H_nf)(s)}ds 
&=\prod_{k=0}^{n-1}\int_{(0,1)}e^{it \left(\sigma_{\frac1n}f\right)(s)}ds\\
&=\left(\int_0^{\frac{m}{n}}e^{itf(ns)}ds+\int_{\frac{m}{n}}^1 1ds \right)^n\\
&=\left(\frac1n\int_0^m e^{itf(s)} ds +1-\frac{m}{n}\right)^n\\
&=\left(1 + \frac1n\int_0^m e^{itf(s)} ds -\frac{m}{n} +\frac1n\int_m^\infty \left( e^{it \cdot 0}-1\right)  ds \right)^n\\
&=\left(1+\frac1n\int_0^{\infty}(e^{itf(s)}-1)ds\right)^n,
\end{align*}
where $\Omega =\Pi_{k=0}^\infty (0,1)$ be the (infinite dimensional) hypercube equipped with the product Lebesgue measure 
$dm^\infty $. 
For every fixed $t\in\mathbb{R},$ we have
$$\left(1+\frac1n\int_0^{\infty}\left(e^{itf(s)}-1\right)ds \right)^n\to \exp\left(\int_0^{\infty}\left(e^{itf(s)}-1\right)ds\right).$$
Thus,
$$\int_{\Omega}e^{it(H_nf)(s)}ds \to \exp\left(\int_0^{\infty}\left(e^{itf(s)}-1\right)ds\right)
\stackrel{\eqref{char for comm K}}{=} \int_0^1e^{it(Kf)(s)}ds .$$
Since convergence of characteristic functions yields the convergence in distributions~\cite[Theorem 6.2.1]{Borovkov}, the claim follows.

Since the space $M_{\psi}$ possesses the Fatou property\footnote{
It is known that $M_{\psi}$ 
is maximal in the terminology of \cite[p.104]{KPS}.
The Fatou property is 
  equivalent with maximality, see \cite[p.118]{LT2}.}, it follows that
$$\left\|Kf\right\|_{M_{\psi}}\leq\liminf_{n\to\infty}\left\|\sum_{k=0}^{n-1}1^{\otimes k}\otimes \sigma_{\frac1n}f\otimes 1^{\otimes\infty}\right\|_{M_{\psi}},$$
see e.g. \cite[Appendix C]{Astashkinbook}.
By  \cite[Proposition 18]{ASZ}, we have
\begin{align*}
  \left\|
\sum_{k=0}^{n-1}1^{\otimes k}\otimes \sigma_{\frac1n}f\otimes 1^{\otimes\infty}\right\|_{M_{\psi}}&\leq c_{{\rm abs}}\left\|\bigoplus_{k=0}^{n-1}1^{\otimes k}\otimes \sigma_{\frac1n}f\otimes 1^{\otimes\infty}\right\|_{L_2\cap L_{\infty}}\\
&=c_{{\rm abs}}\left\|\bigoplus_{k=0}^{n-1}\sigma_{\frac1n}f\right\|_{L_2\cap L_{\infty}}=c_{{\rm abs}}\left\|f\right\|_{L_2\cap L_{\infty}}.
\end{align*}
This completes the proof. 
\end{proof}

The following lemma is one of the key instruments to construct an isomorphism from a symmetric space over the $II_\infty$-factor to some symmetric space over the $II_1$-factor.  

\begin{lem}\label{simple function estimate} Let $E(0,1)$ be a strongly symmetric function space. If 
 the operator $K$ is bounded from $E(0,1)$ to $E(0,1)$ (equivalently, $K^{AS}$ is bounded from $E(0,1)$ to $E(\Omega)$), then there exists a constant 
 $c_E$ depending on $E$ only such that 
$$c_E^{-1}\left\|f\right\|_{Z_E^2}\leq \left\|Kf\right\|_E\leq c_E\left\|f\right\|_{Z_E^2}$$
for every symmetrically distributed $f\in Z_E^2(0,\infty)$ supported on a bounded interval $(0,m)$ for some $m<\infty $.
\end{lem}
\begin{proof} Choose $g$ such that 
\begin{enumerate}
\item $g$ is equidistributed with $f;$
\item $g\chi_{_{[0, 1)}}$ and $g\chi_{_{[1, \infty )}}$ are  symmetrically distributed;
\item $|g|\chi_{_{[0, 1)}}$ and $|g|\chi_{_{[1, \infty )}}$  are  equidistributed with $\mu(f)\chi_{_{[0, 1)}}$ and $\mu(f)\chi_{_{[1, \infty )}}$. 
\end{enumerate}
Set $g_1=g\chi_{(0,1)}$ and $g_2=g\chi_{_{[1,\infty)}}.$ Clearly, $Kf$ is equidistributed with $Kg=Kg_1+Kg_2$ (see Lemma \ref{equimeasurability lemma} and \cite[Corollary 33(i)]{JSZ}), where $Kg_1$ and $Kg_2$ are mutually independent. 

By the  triangle inequality, we have 
\begin{align*}
\left\|Kf\right\|_E&=\left\|Kg \right\|_E \leq \left\|Kg_1\right\|_E+ \left\|Kg_2\right\|_E\leq \left\|Kg_1\right\|_E+\left\|{\rm id}\right\|_{M_{\psi}\to E}\left\|Kg_2\right\|_{M_{\psi}},
\end{align*}
where $\left\|{\rm id}\right\|_{M_{\psi}\to E}$ is finite\footnote{Recall that $K$ is bounded from $E(0,1)$ into $E(\Omega)$, which implies that $K$ maps $L_\infty (0,1)$ into $E(\Omega)$. Therefore, by \cite[Lemma 12]{ASZ}, $\left\|{\rm id}\right\|_{M_{\psi}\to E}$ is finite. }. 
By Lemma \ref{mpsi estimate}, we have
\begin{eqnarray*}
& &   \left\|Kf\right\|_E\\&\leq&\left\|K\right\|_{E\to E}\left\|g_1\right\|_E+c_{{\rm abs}}\left\|{\rm id}\right\|_{M_{\psi}\to E}\left\|g_2\right\|_{L_2\cap L_{\infty}}\\
&=&\left\|K\right\|_{E\to E} \left\|\mu(f)\chi_{(0,1)}\right\|_E+c_{{\rm abs}}\left\|{\rm id}\right\|_{M_{\psi}\to E}\max\left\{\mu(1;f),\left\|\mu(f)\chi_{_{(1,\infty)}}\right\|_2\right\}\\
&\stackrel{\eqref{defZE2}, \eqref{normalized}}{\le}& \left\|K\right\|_{E\to E} \left\|f\right\|_{Z_E^2}+ c_{{\rm abs}}\left\|{\rm id}\right\|_{M_{\psi}\to E} 
\left\|f\right\|_{Z_E^2}
\\
&=&  \left(
\left\|K\right\|_{E\to E}+ c_{{\rm abs}}\|{\rm id}\|_{M_{\psi}\to E}\right)\left\|f\right\|_{Z_E^2}.
\end{eqnarray*}

To prove the converse inequality, note that
$Kg_1$ and $Kg_2$ are mutually independent and mean zero.
%
Since $E(0,1)$ is strongly symmetric, it follows from Lemma~\ref{braverman lemma}  that
$$\left\|Kf\right\|_E=\left\|Kg\right\|_E\geq\left\|Kg_1\right\|_E.$$
Since $Kg_1$ is equimeasurable with $K^{AS}g_1$ (see Section \ref{S:k}), it follows from the definition of the operator $K^{AS}$  that
$$\mu\left(Kg_1\right)=\mu\left(K^{AS}g_1\right)\geq\sigma_{e^{-1}}\mu(g_1)\geq\sigma_{\frac13}\mu(g_1).$$
Thus,
$$\left\|Kf\right\|_E\geq \left\|\sigma_{\frac13}\mu(g_1)\right\|_E\geq\frac13
\left\|g_1\right\|_E=\frac13\left\|\mu(f)\chi_{(0,1)}\right\|_E.$$
On the other hand, we have $E(0,1)\subset L_1(0,1)$ (see e.g. \cite{KPS}, \cite{DP2} and \cite[Theorem 4.4.6]{DPS}) and 
$$\left\|Kf\right\|_E\geq \left\|{\rm id}\right\|_{E\to L_1}^{-1}\left\|Kf\right\|_1.$$
It follows from   Lemma 38 in \cite{JSZ} that
$$\left\|Kf\right\|_1\geq c_{{\rm abs}}'\left\|f\right\|_{L_1+L_2}$$
for every symmetrically distributed $f\in L_1(0,\infty)$ supported on a bounded interval. Thus,
\begin{align*}\left\|Kf\right\|_E&~\qquad \geq\qquad \max\left\{
\frac13\left\|\mu(f)\chi_{(0,1)}\right\|_E,\left\|{\rm id}\right\|_{E\to L_1}^{-1}c_{{\rm abs}}'\left\|\mu(f)\right\|_{L_1+L_2}\right\}\\
&
\stackrel{\mbox{\tiny \cite[Example 1]{Mali}}}{\gtrsim}\max\left\{
\frac13\left\|\mu(f)\chi_{(0,1)}\right\|_E,\left\|{\rm id}\right\|_{E\to L_1}^{-1}c_{{\rm abs}}'\left\|\mu(f)\chi_{_{(1,\infty )}}\right\|_{ L_2}\right\}\\
& ~\qquad \gtrsim \qquad   \norm{f}_{Z_E^2}.
\end{align*}
The proof is complete. 
\end{proof}

The following lemma was proved in the special setting of $L_p$-spaces in \cite[Lemma 38]{JSZ} and \cite[Proposition A.4]{HJSZ}.
Recall the definitions of  $\mathscr{R}_{_\cM}:L_1(\cM,\tau)\to L_1\left(
\cN_{_{L_\infty (0,1) } \bar{\otimes } \cM  }
\right)$  and $\mathscr{R}:L_1(\mathcal{R}\bar{\otimes}\mathcal{L}(H))\to L_1(\cR)$ (see \eqref{defRM} and \eqref{defR}).

\begin{lem}\label{from js lemma} Let $\cM$ be a semifinite von Neumann algebra equipped with a semifinite faithful normal trace $\tau$. 
Let $E(0,1)$ be a strongly 
symmetric function space. 
If $K$ is bounded from $E(0,1)$ to $E(0,1)$, then
\begin{enumerate}
\item
$$\left\| \mathscr{R}_{_{\cM }} x \right\|_E\approx_E \left\|x\right\|_{Z_E^2}$$
for every  $x\in Z_E^2(\cM,\tau)$ with $\tau$-finite supports. 
Moreover,   
 $\mathscr{R}_{_{\cM }}   $ extends to an isomorphic embedding from $Z_E^2(\cM,\tau )$ into $E\left(
\cN_{_{L_\infty (0,1)} \bar{\otimes}\cM }\right).$
\item $$\|\mathscr{R}x\|_E\approx_E\|x\|_{Z_E^2}$$
for every  finitely supported $x\in Z_E^2(\mathcal{R}\bar{\otimes}\mathcal{L}(H)).$ Moreover,   $\mathscr{R}$ extends to an isomorphic embedding from $Z_E^2(\mathcal{R}\bar{\otimes}\mathcal{L}(H))$ into $E(\mathcal{R}).$
\end{enumerate}

\end{lem}
\begin{proof} 
It suffices to prove (1). By a standard argument, one may assume that $x$ is self-adjoint (see e.g. the proofs of \cite[Lemma 38]{JSZ} or \cite[Proposition A.4]{HJSZ}). Let $f$ be a function on $(0,\infty)$ equidistributed with $r\otimes x.$ It follows   that $\mathscr{R}(x)=\mathscr{K}(r\otimes x)$ is equidistributed with $Kf$ (see Section \ref{S:k}). The assertion follows from Lemma~\ref{simple function estimate}.
\end{proof}

Observe  the following property of symmetric spaces  affiliated hyperfinite  type $II$ factors. 

\begin{prop}\label{prop:square}
Let $\cM$ be the hyperfinite $II_1$-factor $\cR$ equipped with a  faithful normal tracial state $\tau$ (or $\cR\bar{\otimes} \mathcal{L}(  H  )$ equipped with a semifinite faithful normal trace $\tau$). For any symmetric space $E(\cM,\tau)$, we have 
$$E(\cM,\tau)\approx E(\cM,\tau)\oplus E(\cM,\tau).$$
As a consequence, $E(L_\infty(0,1) \bar{\otimes} \cM, m\otimes \tau) \approx E(\cM,\tau)$. 
\end{prop}
\begin{proof}
Letting $\ell_\infty $ be equipped with a faithful normal tracial state $\nu$ such that $\nu( \chi_{e_k  }) =\frac{1}{2^k}$, there exists a Banach space $X$ such that 
$$E(\ell_\infty  \bar{\otimes}  \cR)\oplus X\approx E(\cR) .$$
Therefore, we have 
$$E(\cR)\approx E(\ell_\infty \bar{\otimes}  \cR)\oplus X \approx   E(\cR)\oplus   E(\ell_\infty  \bar{\otimes}  \cR)   \oplus X \approx E(\cR)\oplus E(\cR).$$
Similarly, there exists a Banach space $Y$ such that 
$$E (\ell_\infty \bar{\otimes} \mathcal{R}\bar{\otimes}  \mathcal{L}(H)  ) \oplus Y \approx E(   \mathcal{R}\bar{\otimes} \mathcal{L}(H)  )  .$$
Therefore, we have 
\begin{align*}
E(   \mathcal{R}\bar{\otimes}  \mathcal{L}(H)  )  &  \approx  E(\ell_\infty  \bar{\otimes}   \mathcal{R}\bar{\otimes} \mathcal{L}(H)  ) \oplus Y \\
 & \approx    E (   \mathcal{R}\bar{\otimes}  \mathcal{L}(H)  )   \oplus    E (\ell_\infty \bar{\otimes}   \mathcal{R}\otimes\mathcal{L}(H)  )  \oplus Y \\
&\approx  E  (   \mathcal{R}\bar{\otimes} \mathcal{L}(H)  )  \oplus  E  (   \mathcal{R}\bar{\otimes}  \mathcal{L}(H)  )  .
\end{align*}

Observe that $$E(L_\infty(0,1) \bar{\otimes} \cM, m\otimes \tau)\hookrightarrow E(\cM,\tau) ,~E(\cM,\tau)  \hookrightarrow E(L_\infty(0,1) \bar{\otimes} \cM, m\otimes \tau)  $$ as complemented subspaces and 
$$E(L_\infty(0,1) \bar{\otimes} \cM, m\otimes \tau) \approx E(L_\infty(0,1) \bar{\otimes} \cM, m\otimes \tau)\oplus E(L_\infty(0,1) \bar{\otimes} \cM, m\otimes \tau) .$$
The last assertion follows  from Pelczynski's Decomposition Principle \cite[Theorem~2.2.3]{AK}, see also \cite{LT1}. 
\end{proof}

\begin{lem}\label{jinghao inverse lemma}  
Let $E(0,1)$ be a symmetric space in the sense of Lindenstrauss and Tzafriri.
Let    $\cM$ be a   semifinite von Neumann algebra   equipped with a semifinite faithful normal trace $\tau$. 
If both $E(0,1)$ and $E(0,1)^{\times}  $     possess  the Kruglov property, then
$$ \left(\mathscr{R}_{_\cM} \Big|_{Z_{E^\times}^2 \to E^\times}\right)^*\circ  \left(\mathscr{R}_{_\cM} \Big|_{Z_{E}^2 \to E}\right)={\rm id}.$$
\end{lem}
\begin{proof} 
For simplicity, we only consider the special case when $\cM=\cR\bar{\otimes}\mathcal{L}(H)$. In particular, we write  $\mathscr{R}_{_\cM} =\mathscr{R}$.

Applying Lemma \ref{from js lemma} to the space $E(0,1),$ we infer that the mapping 
$$\mathscr{R}:Z_E^2(\mathcal{R}\bar{\otimes} \mathcal{L}(H))\to E(\mathcal{R})$$
 is an isomorphic embedding.
 
By Lemma \ref{hilbertian}, 
for any finitely supported elements $x,y\in 
\mathcal{R}\bar{\otimes} \mathcal{L}(H),$ we have
$$\tau \left(
 \left(\mathscr{R} \right)(x)\cdot  \left(\mathscr{R} \right)(y) \right)=\left(\tau\otimes{\rm Tr}\right)(xy).$$

Applying Lemma \ref{from js lemma} to the space $E(0,1) ^{\times},$ we infer that the mapping $$\mathscr{R}:Z_{E^{\times}}^2(\mathcal{R}\bar{\otimes}\mathcal{L}(H))\to E^{\times}(\mathcal{R})$$ is an isomorphic embedding. By duality, we have that $$ \left(\mathscr{R}\Big|_{Z_{E^\times}^2 \to E^\times }\right) ^* :(E^{\times}(\mathcal{R}))^*\to(Z_{E^{\times}}^2(\mathcal{R}\bar{\otimes}  \mathcal{L}(H)))^*$$
 is a bounded operator.  Thus,
\begin{align}\label{eq1}\begin{split}
  (\tau\otimes{\rm Tr})(xy) =\langle\mathscr{R}(x),\mathscr{R}(y)\rangle&=\left\langle x, \left(\mathscr{R}\big|_{Z_{E^\times}^2 \to E^\times}\right)^*(\mathscr{R}y)\right\rangle   \\
&=\left\langle x,\left( \left(\mathscr{R}\big|_{Z_{E^\times}^2 \to E^\times}\right) ^*\circ\mathscr{R}\right)(y) \right\rangle.
\end{split}
\end{align}
On the other hand, since $\mathscr{R}:Z_{L_1}^2(\cR\bar{\otimes}\mathcal{L}(H))\to L_1(\cR)$ is bounded (see Lemma \ref{from js lemma} or \cite[Lemma 38]{JSZ}) and $L_1(0,1)\supset E(0,1)^\times$, it follows 
from  Proposition \ref{observation}  that 
$$ \left(\mathscr{R}\Big|_{Z_{L_1}^2 \to L_1}\right)^* 
=  
\left (\mathscr{R}\Big|_{Z_{E^\times }^2 \to E^\times}\right)^*
$$
on $L_1(\cR)^\times = \cR$. 
Moreover,  
 \begin{align}\label{intersection}\begin{split}
 \left(\mathscr{R}\big|_{Z_{E^\times}^2 \to E^\times}\right)^*(y)&= 
 \left(\mathscr{R}\big| _{Z_{L_1}^2 \to L_1}\right)^*(y) \\
 &\in Z_{L_1}^2(\cR\bar{\otimes}\mathcal{L}(H))^*  =
Z_{L_1}^2(\cR\bar{\otimes}\mathcal{L}(H))^\times  
\end{split}
\end{align} for any $y\in L_1(\cR)^\times = \cR $.
Observe that 
 $Z_{E^{\times}}^2(\mathcal{R}\bar{\otimes}\mathcal{L}(H))$ and  $  E^{\times}(\mathcal{R}) $ have the Fatou property \cite[Theorem 27]{DP2}.
 Since 
$ \left(\mathscr{R}\Big|_{Z_{E^\times}^2 \to E^\times }\right) ^*(y )  \in Z^2_{E^\times}(\cR\bar{\otimes }\mathcal{L}(H)) ^\times $ for all $y\in \cR $ (see \eqref{intersection}), it follows from 
Lemma~\ref{fact} 
 that    
\begin{align}\label{kothe}
 \left(\mathscr{R}\Big|_{Z_{E^\times}^2 \to E^\times }\right) ^* :E(\mathcal{R})^{\times \times}  \to Z_{E^\times }^2(\mathcal{R}\bar{\otimes} \mathcal{L}(H))^\times .
\end{align}
By \eqref{eq1} and \eqref{kothe}, we obtain that, for every finitely supported element $y\in \mathcal{R}\bar{\otimes}  \mathcal{L}(H),$   
\begin{align}\label{finite support}
\left(
 \left(\mathscr{R}\Big|_{Z_{E^\times}^2 \to E^\times}\right)^*\circ\mathscr{R}\right)(y)=y.
\end{align}
 
{\bf The case when $E(0,1)$ is separable}.
It follows from \eqref{finite support} and 
the boundedness of $ \left(\mathscr{R}\Big|_{Z_{E^\times}^2 \to E^\times}\right)^*\circ\mathscr{R}$  that  $$ \left(\mathscr{R}\Big|_{Z_{E^\times}^2 \to E^\times}\right) ^*\circ\mathscr{R}={\rm id}$$
on $Z_{E }^2(\mathcal{R}\bar{\otimes} \mathcal{L}(H)) $; 

{\bf The case when $E(0,1)$ has the Fatou property}.
By \eqref{kothe} (taking $E$ instead of $E^\times $) and the Fatou property of $E(0,1)$, we have
\begin{align}\label{Re}
  \left(\mathscr{R}\Big|_{Z_{E }^2 \to E  }\right) ^* :E(\mathcal{R})^\times   \to Z_{E ^\times  }^2(\mathcal{R}\bar{\otimes} \mathcal{L}(H)) .
\end{align}By Lemma \ref{fact2}, $\mathscr{R}$ is $\sigma(Z_E^2 , (Z_E^2)^\times )- \sigma (E,E^\times)$-continuous.  

%
%

Let 
$T:= \left(\mathscr{R}\Big|_{Z_{E^\times }^2 \to E ^\times  }\right) ^*:E(\mathcal{R})   \to Z_{E  }^2(\mathcal{R}\bar{\otimes} \mathcal{L}(H))  $. 
Observing  that  for any $x\in Z_{E^\times} ^2( \cR\bar{\otimes } \mathcal{L}(H) )$ and  $y\in E(\cR)  $, we have
$$
\tau
\left( \mathscr{R}(x)\cdot y \right)
=
\langle 
 \mathscr{R } x, y\rangle
= 
\left\langle  x,  \left(\mathscr{R}\big|_{Z_{E^\times}^2 \to E^\times}\right) ^* y\right\rangle  
 = \left\langle        x, T  y\right\rangle  
 = \left\langle    T  ^* (x),   y\right\rangle  . $$
This implies that $T^*(x)=\left(
 \left(\mathscr{R}\big|_{Z_{E^\times}^2 \to E^\times}\right) ^* \right)^* (x) \in E(\cR)^\times $, see e.g. \cite[Theorem 37]{DP2}. 
By 
Lemma  \ref{fact2}, $T=\left(\mathscr{R}\big|_{Z_{E^\times}^2 \to E^\times}\right) ^* $ is $(E,E^\times )-(Z_E^2 ,(Z_E^2)^\times)$-continuous.


Since finitely supported bounded elements are  dense in   $Z^2_E(\cR\bar{\otimes}\mathcal{L}(H))$   with respect to    $\sigma\left(Z^2_E,(Z^2_E)^\times \right)$-topology, it follows from \eqref{finite support}  that 
 $$ \left(\mathscr{R}\Big|_{Z_{E^\times}^2 \to E^\times}\right) ^*\circ\mathscr{R}={\rm id}$$
on $Z_E^2(\mathcal{R}\bar{\otimes} \mathcal{L}(H))$. 
\end{proof}
 
  Now, we present the proof for Theorem \ref{main}. 
\begin{proof}
The proofs for the cases of general semifinite von Neumann algebras $\cM$ and the hyperfinite $II_\infty $ factor $\cR\bar{\otimes} \mathcal{L}(H)$ are the same. 
Below, we present the proof for the case of  $\cR\bar{\otimes}  \mathcal{L}(H)$.

Set
$$X_0=Z_E^2(\mathcal{R}\bar{\otimes} \mathcal{L}(H)), ~X_1=E(\mathcal{R}),$$
$$T_0=\mathscr{R} , ~T_1= \left(\mathscr{R}\Big|_{Z_{E^\times}^2 \to E^\times }\right) ^*.$$ 
By Lemma \ref{jinghao inverse lemma}, the spaces $X_0$ and $X_1$ and the mappings $T_0$ and $T_1$ satisfy the assumptions in Lemma \ref{standard lemma}. Applying Lemma \ref{standard lemma} to the spaces $X_0$ and $X_1$ and to the mappings $T_0$ and $T_1$,
 we obtain that $Z_E^2(\mathcal{R}\bar{\otimes} \mathcal{L}(H))$ is isomorphic to a complemented subspace in $E(\mathcal{R}).$

The mapping $$x\to x\otimes e_{0,0}$$ is an isomorphic embedding of $E(\mathcal{R})$ into $Z_E^2(\mathcal{R}\bar{\otimes} \mathcal{L}(H)),$ whose image is complemented, where $e_{0,0}$ stands for the $(0,0)$-th matrix unit. 
 The last assertion follows  from Pelczynski's Decomposition Principle \cite[Theorem 2.2.3]{AK},  by observing that (see Proposition \ref{prop:square})
$$E(\cR)\oplus E(\cR) \approx E(\cR)$$   and $$Z_E^2(\mathcal{R}\bar{\otimes} \mathcal{L}(H))\oplus Z_E^2(\mathcal{R}\bar{\otimes} \mathcal{L}(H))\approx Z_E^2(\mathcal{R}\bar{\otimes} \mathcal{L}(H)).$$
\end{proof}

\begin{rmk}\label{KAScommutative}
Note that the Kruglov operator $\mathscr{K}_{_{L_{\infty}(0,\infty)}}$ maps  elements from  $L_1(0,\infty)$ to $L_1(\cA)$ for some abelian von Neumann algebra $\cA$ equipped with a faithful normal tracial state. 
Indeed, if $f,g\in L_1(0,\infty)$ are real-valued, then, by \cite[Corollary 33(iv)]{JSZ}, $\mathscr{K}_{_{L_{\infty}(0,\infty)}}f$ and $\mathscr{K}_{_{L_{\infty}(0,\infty)}}g$ are commuting self-adjoint operators affiliated with $\cN_{_{L_{\infty}(0,\infty)}}.$ Set now $\cA$ to be the von Neumann algebra generated by (the spectral projections of) elements $\{\mathscr{K}_{_{L_{\infty}(0,\infty)}}f\}_{f\in L_1(0,\infty)}.$ 
\end{rmk}

 \appendix

\section{Properties of adjoint operators}

The following fact is folklore in the commutative setting, see e.g. \cite[Lemma 1 and Theorem 1]{AM} and \cite{Gribanov}. 
\begin{lem}\label{fact}
Let $\cM$ and $\cN$ be  von Neumann algebras equipped with semifinite faithful normal traces $\tau$ and $\nu$ respectively. 
Assume that $E_1(\cM,\tau) $ and $E_2(\cN,\nu )$ are strongly  symmetric spaces.
Assume that  $$T:E_1(\cM,\tau)\to E_2(\cN,\nu)$$ is  a bounded operator such that $T^*:(L_1\cap L_\infty)  (\cN,\nu)\to E_1(\cM,\tau)^\times$,  
  where $T^*$ stands for the conjugate operator  of $T$.

Assume, in addition that, $\cN$ is $\sigma$-finite or $E_2(\cN,\nu)^\times \subset S_0(\cN,\nu)$. 
If both $E_1(\cM,\tau)$ and $E_2(\cN,\nu )$ have the Fatou property, then 
  $$T^*(y) \in E_1(\cN,\nu)^\times$$ for all $y \in  E_2(\cN,\nu)^\times $. 
\end{lem}
\begin{proof}
 By assumption, for any $y\in (L_1\cap L_\infty)  (\cN,\nu)   $,  we have
$$\tau(T( x ) y ) = \tau(x T^*(y)),~\forall x\in E_1(\cM,\tau). $$ 
Note that \begin{align*}
\left|\tau(x T^*(y)) \right| &=\left| \tau(T(x)y ) \right|\\
& \le \norm{Tx}_{E_2}\norm{y}_{E_2 ^\times } \\
&\le \norm{T}_{E_1\to E_2} \norm{x}_{E_1} \norm{y}_{E_2 ^\times },~\forall x\in E_1(\cM,\tau) ,
\end{align*}
i.e.,   $\norm{T^*(y)}_{E_1^\times}  \le  \norm{T}_{E_1\to E_2}   \norm{y}_{E_2 ^\times } $.

Since  $\cN$ is $\sigma$-finite or $E_2(\cN,\nu)^\times \subset S_0(\cN,\nu)$, it follows that  
for any $y\in E_2(\cN,\nu)^\times $, there exists a sequence 
$y_n\in (L_1\cap L_\infty )(\cN,\nu) $ such that 
$y_n\to y$ as $n\to \infty $ in the $\sigma(E_2^\times , E_2^{ \times\times} )$-topology, see e.g.  \cite[Proposition 3.0.9]{FGHS}. 
Since $E_2(\cN,\nu)$ has the Fatou property, it follows from \cite[Theorem 32]{DP2} that $$y_n\to y \mbox{ as } n\to \infty $$ in the $\sigma(E_2^\times , E_2)$-topology. 
Therefore, we have
$$\tau(T( x ) y )  =\lim_{n\to \infty} \tau(T(x)y_n) =
\lim_{n\to \infty} \tau(x T^*(y_n)) $$ 
 for any $x\in E_1(\cM,\tau)$.
This implies that 
$T^*(y_n)$ is a $\sigma(E _1^\times, E_1)$-Cauchy sequence in $E_1 (\cM,\tau)^\times$.
Since $E _1 (\cM,\tau)$ and 
 $E _1 (\cM,\tau)^\times$ have the Fatou property~\cite[Theorem 27]{DP2}, it follows that $E _1 (\cM,\tau)^\times$ is $\sigma(E_1^\times, E_1 )$-sequentially complete\cite[Proposition 3.1]{DK}.
Hence, there exists $z\in E_1 (\cM,\tau)^\times$ such that $$T^*(y_n)\to z$$ as $n\to \infty $ in the $\sigma(E_1^\times, E_1)$-topology.
In particular, 
$$ ( T^*(y) )(x) =\tau(T(x)y)=\lim_{n\to \infty} \tau(x T^*(y_n))= \tau(x z)$$
for all $x\in E_1(\cM,\tau)$. 
Hence, $T^*(y)=z \in E_1(\cM,\tau)^\times$. 
\end{proof}

We also need the following observation. 
\begin{prop}\label{observation}
Let $T$ be a bounded linear operator from a strongly symmetric space $E(\cM,\tau)$ into another one $F(\cN,\nu)$.
Let $E_1(\cM,\tau)\subset E(\cM,\tau)$ and $F_1(\cN,\nu)\subset F(\cN,\nu)$ be two strongly symmetric spaces with $$T:E_1(\cM,\tau)\to F_1(\cN,\nu).$$ 
If $E(\cM,\tau)$ has order continuous norm, then for any 
$y\in F(\cN,\nu)^\times $, we have $\left(
T|_{E_1\to F_1 }\right)^*(y) \in E_1(\cM,\tau) ^\times ,$   $  \left(T|_{E\to F}\right)^*(y)\in E(\cM,\tau)^\times $ and 
$$(T|_{E\to F})^* (y)  = (T|_{E_1 \to F_1})^*(y).$$
\end{prop}
\begin{proof}
For any $y \in F(\cN,\nu)^\times $ and $x\in E_1(\cM,\tau)\subset E(\cM,\tau)  $, we have 
\begin{align}\label{TEF}\begin{split}
\left(T|_{E\to F}\right)^*(y) (x)  & = \tau\left(y T|_{E\to F}(x)\right) \\
                             & = \tau
\left(y T|_{E_1\to F_1}(x)\right) \\
                             & =\left(
T|_{E_1\to F_1 }\right)^*(y) (x) .
                        \end{split}
\end{align}

Observe that 
for any $y\in   F(\cN,\nu)^\times \stackrel{\tiny \mbox{\cite[Prop. 4.3.12]{DPS}}}{\subset} F_1(\cN,\nu)^\times $ and  any net $\{x_\alpha\}\in 
  E_1(\cM,\tau)$ with $x_\alpha\downarrow 0$,  we have
\begin{align*}  
 \left(
T|_{E_1\to F_1 }\right)^*(y) (x_\alpha) & =
    |\tau( T(x_\alpha)   y   )|\\
   & \le \norm{ T (x_\alpha)   }_{F } \norm{y}_{F^\times}\\
 & \le 
\norm{  T     }_{E\to F }
\norm{x_\alpha}_{E}\norm{y}_{F^\times  } 
\to_\alpha 0. 
\end{align*}  
Hence, 
by \cite[Theorem 5.2.9]{DPS}, we have 
\begin{align}\label{TE1F1} 
\left(
T|_{E_1\to F_1 }\right)^*(y) \in E_1(\cM,\tau) ^\times  
\end{align}
for any $y\in F(\cN,\nu)^\times$.
Since $E(\cM,\tau)$ has order continuous norm, it follows that  
$$ \left(T|_{E\to F}\right)^*(y)\in E(\cM,\tau)^* =E(\cM,\tau)^\times 
$$ 
for any $y\in F(\cN,\nu)^\times$.
This together with \eqref{TEF} and \eqref{TE1F1} completes
the proof.  
\end{proof}

 The following lemma was established in  \cite[Lemma 3.0.5]{FGHS} (see also \cite{AC17,AC18}), whose commutative counterpart can be found in 
  \cite{KR2}. 
\begin{lem} 
\label{fact2}
Let $\cM_1$ and $\cM_2$ be  semifinite von Neumann algebras equipped with semifinite faithful normal traces $\tau_1$ and $\tau_2$, respectively.
Let $E(\cM_1,\tau_1)$ and $F(\cM_2,\tau_2)$ be two noncommutative strongly  symmetric  spaces.   
A bounded linear operator $T$ from $E(\cM_1,\tau_1)$ to $F(\cM_2,\tau_2)$  is  $\sigma(E,E^\times)-\sigma(F,F^\times)$-continuous if and only if $T^*$ maps $F(\cM_2,\tau_2)^\times $ into $E(\cM_1,\tau_1)^\times$. 
\end{lem}

\end{document}